\documentclass[12pt,sans]{article}
\usepackage{amsmath}
\usepackage{amsfonts}
\usepackage{amsthm}
\usepackage{amssymb}
\usepackage[sans]{dsfont}
\usepackage[final]{graphicx}
\usepackage{mathrsfs}
\usepackage{nicefrac}
\usepackage{sans}
\usepackage[T1]{fontenc}
\usepackage{siunitx}

\usepackage[T1]{fontenc}
\usepackage[ansinew]{inputenc}

\usepackage{natbib}
\bibpunct{(}{)}{;}{a}{,}{,}

\newtheorem{theorem}{Theorem}
\newtheorem{corollary}{Corollary}
\newtheorem{proposition}{Proposition}
\newtheorem{assumptions}{Assumptions}
\newtheorem{remark}{Remark}
\newtheorem{lemma}{Lemma}
\newtheorem{definition}{Definition}

\def\R{{\mathbb R}}
\def\EXP{{\mathbb E}}
\def\P{{\mathbb P}}
\def\Z{{\mathbb Z}}

\def \A{{\mathcal A}}
\def \L{{\mathcal L}}

\newcommand{\ud}[1]{\, \mathrm{d}#1}

\def \d{\partial}
\newcommand{\deriv}[3][]{\frac{\ud^{#1} \hspace{-0.3mm} #2}{\ud{#3}^{#1}}}
\newcommand{\pderiv}[3][]{\frac{\d^{#1} \hspace{-0.1mm} #2}{\d{#3}^{#1}}}

\newcommand{\grad}[1][]{\nabla_{\! #1} }

\newcommand{\IND}[1] {{ \mathds{1}_{ #1 }} }

\def \la{\langle}
\def \ra{\rangle}
\def \vp{\varphi}
\def \eps{\epsilon}

\def \a{\alpha}
\def \b{\beta}

\def \Dom{\mathrm{Dom}}
\def \Ran{\mathrm{Ran}}

\def \geq{\geqslant}
\def \leq{\leqslant}

\newcommand{\Lp}[1]{ L^{#1}_+ }
\newcommand{\Lm}[1]{ L^{#1}_- }

\begin{document}

\title{Continuity of Local Time: An applied perspective}
\author{ Ramirez, Jorge M.\footnote{Department of Mathematics, 
		National University of Colombia.  
		$^\dagger$ Department of Mathematics, Oregon State University.
		$^+$ Corresponding author:
		waymire@math.oregonstate.edu.  }
	\and Thomann, Enrique A.$^\dagger$
	\and Waymire, Edward C.$^{\dagger+}$
}

\author{Jorge M. Ramirez 
\footnote{ Universidad Nacional de Colombia, Sede Medellin, {jmramirezo@unal.edu.co}}, 
Enirque A. Thomann \footnote{Oregon State University, thomann@science.oregonstate.edu}
and Edward C. Waymire \footnote{Oregon State University, waymire@math.oregonstate.edu}}
%
%
\maketitle

\abstract{
Continuity of  local time for Brownian motion ranks among the most notable mathematical results in the theory of stochastic processes.
This article addresses its implications from the point of view 
of applications.  In particular an extension of previous results on an explicit role of continuity of (natural) local time is obtained for applications to recent classes of problems in physics, biology and
finance involving discontinuities in a dispersion coefficient.  The main theorem and its corollary provide physical principles that relate 
macro scale continuity of deterministic quantities to micro scale continuity of the (stochastic) local time. 
}

\section{Introduction}

Quoting from the backcover of the intriguing recent monograph \cite{barn}:  
\begin{quotation}
	Random change of time
	is key to understanding the nature of various stochastic processes and gives rise to interesting mathematical results and insights of
	importance for the modelling and interpretation of empirically observed dynamic processes
\end{quotation}
This point could hardly have been more
aptly made with regard to the perspective of the present paper.  

The focus of this paper is on identifying the way in which continuity/discontinuity properties of certain local times of a diffusive
Markov process  inform interfacial discontinuities in large scale concentration, diffusion coefficients, and transmission rates.  For these purposes one may ignore processes with drift, and focus on discontinuities in  diffusion rates and/or specific rate coefficients. This builds on related work of the authors where the focus was on interfacial effects on other functionals such as occupation time and first passage times; see \citet{ramirez2013review} for a recent survey. We also restrict our study to one-dimensional processes, some extensions to higher dimensions can be found in \cite{ramirez2011multi}. 

Dispersion problems in the physical sciences are often described by a second order linear parabolic equation in divergence form that results from a conservation law applied to the concentration of some substance. A particular class of interest for this paper is that of the one-dimensional Fokker-Plank equation \eqref{EqFokkerPlanck} with discontinuous parameters $ D $ and $ \eta $ and  specified discontinuities in the concentration. More precisely, we consider the solution $ u(t,y) $ to the following problem for  $ y \in \R\setminus I $, $ t\geq 0 $, 
\begin{equation}\label{EqFokkerPlanck}
\eta \pderiv{u}{t} = \pderiv{}{y} \! \left( \tfrac{1}{2}D \, \pderiv{u}{y} \right), \quad
\left[ D \pderiv{u}{y} \right]_{x_j} \!\!\!= 0, \quad \b^+_{j} u(t,x_j^+) = \b^-_{j} u(t,x_j^-), \quad j \in \Z.
\end{equation}
with a prescribed initial condition $u(0,y) = u_0(y)$, $y \in \R$, and under the following assumptions.

\begin{assumptions} \label{assumptions}
	We consider a discrete set of ``interfaces" 
	\begin{equation}\label{Def_I}
	I := \{x_j: \; j\in \Z\}, \quad x_0 := 0, \quad x_j< x_{j+1}, \quad j \in \Z,
	\end{equation} 
	with no accumulation points. The functions $ D $ and $ \eta $ exhibit jump discontinuities only at points in $ I $:
	\begin{align}\label{Conditions_D_eta}
	[D]_{x_j} := D(x_j^+) - D(x_j^-) \neq 0, \quad [\eta]_{x_j} := \eta(x_j^+) - \eta(x_j^-) \neq 0, \quad j \in \Z.
	\end{align}
	We further assume that $ D $ and $ \eta $ are functions of bounded variation in $ \R $ with $ \eta $ being continuous, and $ D $  differentiable in $ (x_j,x_{j+1}) $ for all $ j \in \Z $. Finally, there exist constants $ 0<k<K<\infty $ such that  
	\begin{align}\label{Cond_Deta_bded}
	k \leq D(x) \leq K, \quad k \leq \eta(x) \leq K, \quad x \in \R.
	\end{align}
	The constants $ \{\beta_j^{\pm} \} $ are strictly positive and such that 
	\begin{equation}\label{Cond_beta}
	\sum_{j \in \Z} \frac{\beta_j^+}{\beta_j^-} < \infty.
	\end{equation}
\end{assumptions}

Equation (\ref{EqFokkerPlanck}) is often referred to as a  {\it continuity equation} for the conserved quantity $ \eta(y) u(t,y) $; Fourier's flux law for heat conduction and the corresponding Fick's law for diffusion being among the most notable such occurrences.  The right-side of the pde is the divergence of the diffusive flux $\frac{1}{2} D \pderiv{u}{y}$ of $ u $, and the first bracket condition is \textit{continuity of flux} at the interfaces.

In Jean Perrin's historic determination of Avogadro's number $ N_A $ from observations of individual particle paths submerged in a homogeneous medium, the procedure was clear \citep{perrin1909mouvement}. Einstein had provided a twofold characterization of the diffusion coefficient $ D $ in Fick's law: first as a function of variables at the molecular scale, including $ N_A $; and second, as the rate of growth of the variance of the position of particles in time. This meant that $ D $, and therefore $ N_A $ could be statistically estimated. 

 If one regards (\ref{EqFokkerPlanck}) as a given physical law that embodies certain interfacial discontinuities at points in $ I $, then the question we want to address is \textit{what corresponding features should be specified about the paths of the stochastic process?} Our basic goal is to show that the answer resides in suitably interpreted continuity properties of local time.

In an informal conversation where we posed this question to David Aldous, his reply was that he wouldn't use a pde to model the phenomena in the first place! Of course this perspective makes the question all the more relevant to probabilistic modeling.  The mathematical (probabilistic) tools are clearly available to do this, and much of the objective of this paper is to identify the most suitable way in which to express the stochastic model  in relation to the underlying phenomena. It will be shown that the interfacial conditions at the pde level, can be characterized in terms of the continuity properties of a certain local time process of the associated diffusion. In this regard, the continuity of local time of standard Brownian motion will be seen to indirectly play a pivotal role.

The evolution problem \eqref{EqFokkerPlanck} can be viewed as the \textit{forward equation} $ \pderiv{u}{t} = \L^*u $ for the operator $ \L^*: \Dom(\L^*) \to L^2(\ud y) $ given by
\begin{equation}\label{Def_Lstar}
\L^*f := \frac{1}{\eta} \deriv{}{y}\left(\tfrac{1}{2} D \deriv{f}{y}\right), \quad 
\end{equation}
for functions $ f \in \Dom(\L^*) $ satisfying, besides other decay conditions, that
\begin{equation}\label{Conditions_Lstar}
\left[D f'\right]_{x_j} \!\!\! = 0, \quad \b_j^+ f'(x_j^+) = \b_{j}^- f'(x_j^-), \quad j \in \Z.
\end{equation}

Due to the presence of the coefficient $\eta(y)$, taking the adjoint in $ L^2(\ud y) $ of $\L^*$ does not generally yield an operator $\L$ that generates a positive contraction semigroup on $C_0(\R)$. In fact, integration by parts yields,
\begin{equation}\label{Def_L}
\L g := \deriv{}{x} \left(\tfrac{1}{2} D \deriv{}{x} \left(\frac{g}{\eta} \right) \right)
\end{equation}
and any $ g \in \Dom(\L) $ will satisfy the following interfacial conditions:
\begin{equation}\label{Conditions_L}
 \left[\frac{g}{\eta}\right]_{x_j} \!\!\! = 0, \quad 
 \frac{D(x_j^+)}{\b_j^+} \left(\frac{g}{\eta}\right)'\!\!\!(x_j^+)
 =\frac{D(x_j^-)}{\b_{j}^-} \left(\frac{g}{\eta}\right)'\!\!\!(x_{j}^-).
\end{equation}
We refer to the corresponding evolution problem $ \pderiv{v}{t} = \L v $ as the \textit{backwards equation}.


While physical laws are often formulated on the basis of conservation principles,  not all such models are apriori conceived in conservation form. In fact some may be explicitly formulated as a specification of coefficients via a Kolmogorov backward equation with an operator of the form \eqref{Def_A} below, or directly in the form of a stochastic differential equation;  for a variety of examples of both in the present context, see  \cite{Berkowitz09, hill1995leakage, hoteit2002three, kuo1999experimental} from hydrolgy, \cite{cantrell2004spatial,  mckenzie2009first, ovaskainen2004habitat, ramirez2012population, schultz2005patch} biology and ecology, \cite{nilsen2011no} finance,  \cite{guo2010particle} astrophysics, and  \cite{matano2008upwelling} from physical oceanography. 

To accommodate the broad class of such possible examples, the present paper follows the following general setting. Let $ \L $ be as in \eqref{Def_L} and define the operator $\A$ by
\begin{equation}\label{Def_A}
\A g := \frac{1}{\eta} \L(\eta g) = \frac{1}{\eta} \deriv{}{x} \left(\tfrac{1}{2} D \deriv{g}{x} \right),
\end{equation}
The interfacial conditions satisfied by a function $g \in \Dom(\A) $ follow from \eqref{Conditions_L} and can generally be written in the  form:
\begin{equation}\label{Conditions_A}
\quad [g]_{x_j} = 0, \quad 
\lambda_j g'(x_j^+) = (1-\lambda_j) g'(x_{j}^-), \quad
j \in \Z
\end{equation}
for some values $ \Lambda := \{\lambda_j: j \in \Z\} \subset (0,1) $ for which we further assume the following decay condition (equivalent to \eqref{Cond_beta} under \eqref{Cond_Deta_bded}):
\begin{equation}\label{ConditionLam}
\sum_{j \in\Z} \frac{1- \lambda_j }{\lambda_j} < \infty.
\end{equation}

In Section \ref{Sec_DiffusionX} we construct a Feller process $ X =\{X(t): t\geq 0\} $ with generator $(\A,\Dom(\A))$. The significance of this association is the following: the fundamental solution $ p(t,x,y) $ to the backwards evolution problem $ \pderiv{v}{t} = \A v $ is precisely the transition probability density of $ X(t) $. Namely, for an initial condition $ v(0,x) = v_0(x) $, the solution $ v(t,x) $ can be written as
\begin{equation}\label{Eqn_pSemigroup}
v(t,x) = \int_{\R} p(t,x,y) v_0(y)  \ud y =  \EXP\left[v_0(X(t)) | X(0) = x\right].
\end{equation}
 It follows, in turn, that the fundamental solution to the original forward problem (\ref{EqFokkerPlanck}) is given by 
\begin{equation}\label{Eqn_densities_pq}
q(t,x,y) = \frac{\eta(x)}{\eta(y)} p(t,x,y).
\end{equation}
 This defines the interpretation we will use with respect to the physics of  \eqref{EqFokkerPlanck} as a Fokker-Plank equation, and to relate it with the diffusion process $X$. 

\begin{remark}
 One might note that this is the form of Doob's h-transform under the further constraint that $\eta$ is harmonic with respect to ${\cal A}$.   However this latter condition is generally not appropriate for the physical examples of interest in the present paper.   On the other hand, were it applicable it could provide an altervative approach to our problem via conditioning, e.g. see \citet{perkowski}.
\end{remark}

In order to identify how the interfacial conditions \eqref{Conditions_A} affect the sample paths of the process $ X $, we will look at the behavior of the \textit{natural local time} of $ X $ at points in $ I $. This notion of local time was introduced in \cite{appuhamillage2013skew} as the density of the occupation time operator with respect to Lebesgue measure. Namely for $ A \in  \mathcal{B}(\R) $, $ \ell^X(t,x) $ is a previsible process, increasing with respect to $ t $, such that
\begin{equation}
\int_0^t \IND{A} (X(s)) \ud s = \int_A \ell^X(t,x) \ud x.
\end{equation}

Our main result is the following:

\begin{theorem}\label{Theo_MainResult}
	Suppose $ D, \eta $ satisfy Assumptions \ref{assumptions}. Let $ X $ be the Feller process with infinitesimal generator $ (\A, \Dom(\A)) $ defined by (\ref{Def_A}, \ref{Conditions_A}) and $ \Lambda $ satisfying \eqref{ConditionLam}. Then
	\begin{equation}\label{Eqn_MainResult}
	\frac{\ell^X(t,x_j^+)}{\ell^X(t,x_j^-)} = \frac{\eta(x_j^+)}{\eta(x_j^-)} \frac{D(x_j^-)}{D(x_j^+)} \frac{\lambda_j}{1-\lambda_j}, \quad j \in \Z.
	\end{equation}
	with probability one, for any $ t $ such that $ \ell^X(t,x_j^+) >0 $.
\end{theorem}

\begin{corollary}\label{Cor_MainResult}
	Suppose $ D, \eta $ and $ \{\beta_j^{\pm} : j \in \Z \}$ satisfy Assumptions \ref{assumptions}. Let $ u $ be the solution to the forward equation \eqref{EqFokkerPlanck} and  $ X $ its associated Feller process. Then 
	\begin{equation}\label{Eqn_CorMainResult}
	\frac{\ell^X(t,x_j^+)}{\ell^X(t,x_j^-)} = \frac{\eta(x_j^+)}{\eta(x_j^-)} \frac{\beta_j^-}{\beta_j^+} = 
	\frac{\eta(x_j^+) u(t,x_j^+)}{\eta(x_j^-) u(t,x_j^-)}, \quad j \in \Z.
	\end{equation}
	with probability one, for any $ t $ such that $ \ell^X(t,x_j^+) >0 $.
\end{corollary}

The main principles to be taken from these mathematical results are as
follows.  First, Theorem \ref{Theo_MainResult} and, in particular, its corollary demonstrate how the continuity properties of local time are reflected in the specifications of \eqref{EqFokkerPlanck}  at interfaces.  It is noteworthy that under the continuity of flux condition,  the diffusion coefficient plays no role with regard to continuity of local time.  In particular, its determination would continue to be by statistical considerations of (local) variances along the
lines used by Perrin, while the  jumps in the natural local time are a manifestation of other characteristics of the model. That is, the relative values of  $\eta$ for given values of  $\beta_j^+/\beta_j^-$, or vice-versa, are reflected in the local time behaviors of sample paths at interfacial points. An example is furnished below in which the $\beta$ and $\eta$ parameters are relative manifestations of geometries of  both the medium and the dispersing  individuals.  

Under continuity of flux, the continuity of local time is equivalent to continuity of the conserved
quantity $\eta u$.  In particular if $\eta \equiv 1$, or more generally is  continuous,
the continuity of $ u $ is a manifestation of the continuity of local time.  
These 
connections between continuities at the macro and micro scale are dependent
on the continuity of flux in defining the physical model.

A second principle arises for those contexts in which \eqref{Def_A}
is a prescribed backward equation with $\eta = \beta
\equiv 1$, e.g. as in financial mathematics, \cite{barn}.  From a physical perspective the interface condition is {\it not} a continuity of flux condition,
however continuity of local time occurs if and only if $\lambda_j = {D_j^+\over D_j^- + D_j^+}, j= 1,2,\dots.$

\subsection{Example: piecewise constant coefficients.}

The scope and interest of our result may be illustrated with a single interface example motivated by applications to ecological dispersion in heterogeneous media. Consider a population of erratically moving individuals occupying an infinitely long, two-dimensional duct as depicted in Figure \ref{FigureExample}. Let $ A(y) $ be the cross-sectional area of the duct, and $ 1/\eta(y) $ the biomass of any individual occupying the cross-section at  a distance $ y $ from the interface $ y=0 $. Let $ c(t,(y,\tilde{y})) $ denote the concentration of biomass, which is assumed continuous throughout, $ \tilde y $ denoting the transversal spatial variable. Then $ \eta(y) c(t,(y,\tilde{y})) $ is the concentration of individuals, for which we assume the following  modification to Fick's law: the flux of individuals is given by $ -D\mathbf{I}_{2\times 2}\grad  c $, namely proportional to the gradient of the concentration of biomass. If $ D,A $ and $ \eta $ are taken to be piece-wise constant:
\begin{equation}
D(y) := \begin{cases}
D^+ &\text{if } y> 0\\
D^- &\text{if } y \leq  0,
\end{cases}, \quad 
A(y) := \begin{cases}
A^+ &\text{if } y> 0\\
A^- &\text{if } y \leq 0
\end{cases}, \quad 
\eta(y) := \begin{cases}
\eta^+ &\text{if } y> 0\\
\eta^- &\text{if } y \leq 0
\end{cases},
\end{equation}
then the concentration of biomass per unit length
\begin{align}
u(t,y) = \int_{A(y)} c(t,(y,\tilde{y})) \ud \tilde{y} 
\end{align}
satisfies the following one-dimensional forward problem
\begin{equation}\label{EqModelExamplePDE1}
\pderiv{u}{t} = \frac{1}{\eta}\pderiv{}{y} \! \left( \tfrac{1}{2}D \, \pderiv{u}{y} \right), \quad
\left[ D \pderiv{u}{y} \right]_{0} \!\!\!= 0, \quad \frac{1}{A^-} u(t,0^-) = \frac{1}{A^+} u(t,0^+), 
\end{equation}
which is of the form \eqref{EqFokkerPlanck} with $\beta^\pm_0 = 1/A^\pm $. The corresponding backwards operator $ \A $ is given by \eqref{Def_A} and \eqref{Conditions_A} with $ \lambda_0 = \frac{A^+D^+}{A^+D^+ + A^-D^-} $; let $ X $ denote the diffusion process generated by $ \A $. 

Given an initial biomass distribution $ u_0 $, by virtue of \eqref{Eqn_densities_pq}, the solution to \eqref{EqModelExamplePDE1} can be written in terms of the transition probabilities of $ X $ as:
\begin{equation}
 u(t,y) = \frac{1}{\eta(y)} \int_{\R} u_0(x) \eta(x) \, p(t,x,y) \ud x.
\end{equation}
Namely, $ p(t,x,y) $ is the density function of the \textit{location $ X(t) $ of individuals} that started at $ x $, and the paths of $ X $ may be regarded as a model for the random movement of the individuals.

As in all cases of an operator $ \A $ of the form \eqref{Def_A} with \textit{piece-wise constant} coefficients at a single interface, the associated process $ X $ may be expressed explicitly by re-scaling an appropriate skew Brownian motion  (see \cite{appuhamillage2013skew, appuhamillage2011occupation, ramirez2013review, ouknine1991skew}). Specifically,
\begin{equation}
X(t) = \sqrt{\frac{D\!\left(B^{\a(D,A,\eta)}(t)\right)}{\eta\!\left(B^{\a(D,A,\eta)}(t)\right)}} B^{\a(D,A,\eta)}(t)
\end{equation}
where $B^{\a(D,A,\eta)}$ is skew Brownian motion with transmission probability 
\begin{equation}
\alpha(D,A,\eta) = \frac{A^+ \sqrt{ \eta^+ D^+}}{A^+ \sqrt{ \eta^+ D^+} + A^- \sqrt{\eta^- D^-}}.
\end{equation}

\bigskip

The effect of the interface on the sample paths of the process $ X $ can be readily observed in the stochastic differential equation solved by $X$:
\begin{equation}
 X(t) = \int_0^t \sqrt{\frac{D(X(r))}{\eta(X(r))}} \ud B(r) + \frac{A^+D^+ - A^-D^-}{2 A^+D^+}  L^X_+(t,x),
\end{equation}
where $ B $ denotes standard Brownian motion, and $ \{L^X_+(t,x): t \geq 0\} $ is the right semimartingale local time at $ x $ of the unknown process $ X $. See \cite{revuz_yor1999} for details. For the current example, the right continuous version of natural local time is given by $ \ell_+^X(t,x) := \frac{\eta(x)}{D(x)} L^X_+(t,x) $, and for every $ t>0 $, $ \ell_+^X(t,\cdot) $ is discontinuous at $ x=0 $ with probability one, with a jump characterized by \eqref{Eqn_MainResult}.

\begin{figure}
	\centering
	\includegraphics[scale = 1]{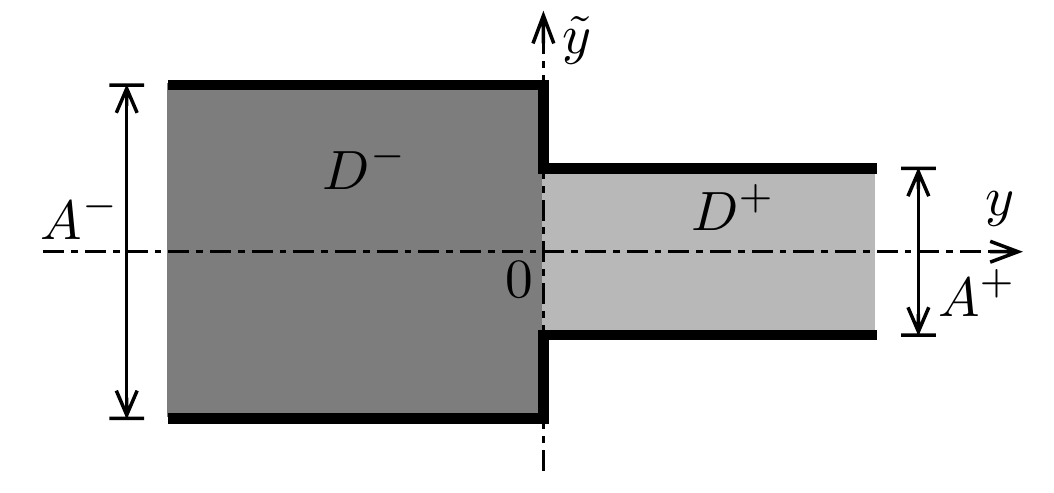}
	\caption{Schematics of an example of a problem leading to a one-dimensional diffusion process with one single interface at $ y=0 $.} \label{FigureExample}
\end{figure}

The nominal effect that a single interface exerts on the particle paths can now be elucidated by looking into further properties of skew Brownian motion, and natural local time. On one hand,  \eqref{Eqn_densities_pq} and the value of $ \alpha(D,\lambda, \eta)$, inform us that for an initial condition concentrated at $ y=0 $, $ u_0 = \delta_0(\ud y) $ in \eqref{EqFokkerPlanck},  the individuals will asymmetrically distribute on either side of the interface:
\begin{align}
\int_0^\infty q(t,0,y) \ud y = \P(X(t) > 0 \big| X(0) = 0) = \frac{\eta^-}{\eta^+} \alpha(D,\lambda, \eta), \quad \text{for all } t>0.
\end{align}

On the other hand, natural local time $ \ell^X(t,x) $ can be related to the time the process spends in a small vicinity of $ x $. Of particular interest is the relative times particles spend at either side of zero. It follows from Corollary \eqref{Cor_MainResult} that: 
\begin{align}\label{Eqn_relLT}
\lim_{ \eps \to 0}\frac{ \int_0^t \IND{(0,\eps)}(X(s)) \ud s}
	{\int_0^t \IND{(-\eps,0)}(X(s)) \ud s} 
	= \frac{\eta^+ A^+}{\eta^- A^-}, \quad t >0, \quad \text{a.s.}
\end{align}

\bigskip

\subsection{Notation and outline}

 The analytical treatment here revolves around functions $ f:\R \to \R $ that are  measurable with respect to the Borel $ \sigma $-algebra $ \mathcal{B}(\R) $. The main function space is $ C_b(\R) $, the space of real valued continuous bounded functions. $ B(\R) $ denotes, in turn, bounded measurable functions on $ \R $. For a measure $ \mu $ on $ (\R, \mathcal{B}(\R)) $, $ L^2(\mu) $ denotes the Hilbert space $ \{f: \R \to \R; \int_{\R} f^2(x) \mu(\ud x) < \infty \} $. The right and left limits of a function at $ x $ are denoted by $ f(x^+) $ and $ f(x^-) $, and their difference is the jump operator $ [f]_x = f(x^+) - f(x^-) $. Whenever defined, the derivative of $ f $ is $ \deriv{f}{x} $ or $ f' $, while $ f_{\pm}'(x) = f'(x^{\pm}) $ denote its left and right derivatives at $ x $.  
 
 For the general theory of one-dimensional diffusion processes used here, we refer the reader to \citet{revuz_yor1999}. A diffusion process consists of the measurable space $ (\Omega, \mathcal{F}) $ and a family of probability measures $ \{\P_x: x \in \R \} $.  A sample path of the process is $ X = \{X(t): t \geq 0\} \in \Omega $. Under $ \P_x $, the paths of $ X $ ``start at $ x $'', namely, $ \P_x(X(0) = x ) = 1 $ for all $ x \in \R $. Only diffusions on $ \R $, with infinite lifetime are considered here. Much of the analysis of such processes is undertaken in terms of their scale and speed measures, which we treat as follows. For $ a \in \R $, the hitting time of $ a $ by $ X $ is $  H^X_a = \inf\{t>0: X(t) = a \} $ and the scale measure $ s(\ud x) $ is characterized by
\begin{equation}\label{Def_sHaHb}
\P_x(H^X_b < H^X_a) = \frac{s((a,x))}{s(a,b)}, \quad a<x<b. 
\end{equation}
Every scale measure in this paper will be absolutely continuous with respect to Lebesgue measure, $ s(\ud x) = s'(x) \ud x $. We also define, without room for confusion, the scale function
\begin{equation}\label{Def_scaleFun}
 s:\R \to \R, \quad s(x) := \int_0^x s'(y) \ud y. 
\end{equation}
The speed measure $ m(\ud x) $ of $ X $ is the unique Radon measure such that 
\begin{equation}\label{Def_mHaHb}
\EXP_x(H_a^X \wedge H_b^X) = \int_\R \frac{\left[s(x\wedge y)-s(a)\right] \left[ s(b)-s(x \vee y)\right]}{s(b)-s(a)} m(\ud y), \quad a<x<b.
\end{equation}
The infinitesimal generator of the process can be written in terms of the speed and scale measures as follows
\begin{equation}\label{Eq_Admds}
\A f = \deriv{}{m} \deriv{}{s} f, \quad f \in \Dom(\A),
\end{equation}
in the sense that,
\begin{equation}
	\deriv{f}{s}(x_2) - \deriv{f}{s}(x_1) = \int_{x_1}^{x_2} \A f(x) m (\ud y).
\end{equation}
Moreover,
\begin{equation}\label{Eq_MartProb}
f(X_t) - f(X_0) - \int_0^t \A f(X(s)) \ud s, \quad  f \in \Dom(\A), \; x \in \R 
\end{equation}
is a martingale. For a given operator $ (\A, \Dom(\A)) $ if \eqref{Eq_MartProb} holds, we say that $ X $ solves the \textit{martingale problem} for $ \A $.

The rest of the paper is organized as follows. In the next section we provide a construction of the diffusion process $ X $ associated to the operator \eqref{Def_A} and identify a stochastic differential equation for which $ X $ is the unique strong solution. In section \ref{Sec_LocalTime} we define three related but different notions of local time, including the \textit{natural local time} and characterize its spatial continuity properties. The proof of the main results \ref{Theo_MainResult} and \ref{Cor_MainResult} follows directly from such characterization.

\section{On the Diffusion $X$} \label{Sec_DiffusionX}

As illustrated in \cite{ramirez2013review}, the derivation of an associated process to an evolution operator \eqref{EqFokkerPlanck} can be achieved in several ways, including the general theory of Dirichlet forms, or by martingale methods. Here, we ``read off'' the speed and scale measures from the backward operator written in the form \eqref{Eq_Admds}, and construct the appropriate process via a stochastic differential equation. A similar approach was carried out in the case of piecewise constant coefficients by \cite{ramirez2011multi}, and will be extended here to the present framework.

Recall Assumptions \ref{assumptions} on $ I $, $ D $ and $ \eta $, and let $ (\A, \Dom(\A)) $ be as in  \eqref{Def_A}, \eqref{Conditions_A}. Recursively define a sequence $\varphi_j, j\in \Z$ by
\begin{equation}\label{Def_phi_j}
\frac{ \vp_j}{\vp_{j-1}} = \frac{D(x_j^+)(1-\lambda_j)}{D(x_j^-) \lambda_j} = \frac{\beta_j^+}{\beta_j^-}, \; j \in \Z, \quad \vp_0 := 1.
\end{equation}
Then the generator $(\A,\Dom(\A))$ given by \eqref{Def_A}  may be equivalently expressed by
\begin{equation}\label{Def_A_final} 
\A f(x) = \frac{\vp_j}{\eta(x)} \deriv{}{x} \left( \frac{D(x)}{2\vp_j} \deriv{f}{x}\right), \quad x \in (x_j, x_{j+1}), \; j \in \Z, 
\end{equation}
acting on functions in $ C_b(\R) $ that are twice continuously differentiable within each $ (x_j,x_{j+1}) $ and such that
\begin{equation}\label{Conditions_A_final}
\frac{D(x_j^-)}{\vp_{j-1}} f'(x_j^-) = \frac{D(x_j^+)}{\vp_{j}} f'(x_j^+), \quad j \in \Z.
\end{equation}

In view of \eqref{Eq_Admds} and \eqref{Def_A_final}, we take $ m(\ud x) := m'(x) \ud x $ and $ s(\ud x) = s'(x) \ud x $ with densities prescribed on $ \R \setminus I $ by
\begin{equation}\label{Def_m_s}
s'(x) = \frac{2 \vp_j}{D(x)}, \quad m'(x) = \frac{\eta(x)}{\vp_j}, \quad x \in (x_j,x_{j+1}), \; j \in \Z.
\end{equation}

The existence of an associated diffusion process is established in the following theorem.

\begin{theorem}\label{Theo_ExistenceX}
	Suppose $ D $ and $ \eta $ satisfy Assumptions \ref{assumptions}, and $ \Lambda $ satisfies \eqref{ConditionLam}. Let $ m $ and $ s $ be measures with densities given by \eqref{Def_m_s}. Then there exist a Feller diffusion $ X = (\Omega, \mathcal{F},\{\P_x: x\in \R \} )$ with speed and scale measures $ m $ and $ s $, respectively, and whose transition probability function $ p(t,x,\ud y) $ is the fundamental solution to the backwards equation $ \pderiv{v}{t} = \A v$  with $ (\A, \Dom(\A)) $ given by \eqref{Def_A_final}, \eqref{Conditions_A_final}. Moreover, $ q $ in \eqref{Eqn_densities_pq} is the fundamental solution to the forward problem \eqref{EqFokkerPlanck}.
\end{theorem}
\begin{proof}
	Note first that the boundedness assumptions on $ D $ and $ \eta $, together with the fact that $ I $ has no accumulation points, make $ m $ and $ s $ Radon measures. Since $ m $ and $ s $ are assumed to have piecewise continuous densities, \eqref{Eq_Admds} takes the form $ \A f(x) = \frac{1}{m'(x^+)}\deriv{}{x}\left(\frac{f'(x)}{s'(x)}\right) $ with $ \Dom(\A) $ being comprised of all functions $ f \in C_b(\R) $ such that $ f'(x)/s'(x) $ is continuous on all of $ \R $ and differentiable in $ \R \setminus I $. Specializing to points in $ I $, this specification is equivalent to the conditions in \eqref{Conditions_A_final}. The range $ \Ran(\A) $ of $ \A $ is contained in $ B(\R) $, and $ \Dom(\A) \times \Ran(\A) $ is a linear subset of $ C_b(\R) \times B(\R)$. The existence of a diffusion process $ X $ that solves the martingale problem for $ (\A, \Dom(A)) $ could now be established under very general conditions  \citep[see for example][]{stroock1979multidimensional}. For our purposes however, Theorem \ref{The_ConstructX} below explicitly constructs a diffusion $ X $ with speed and scale measures given by $ m $ and $ s $, which therefore provides a solution to the martingale problem for $ (\A, \Dom(A)) $ for any $ x \in \R $ \citep[Theorem 3.12, p 308]{revuz_yor1999}. Moreover, conditions \eqref{Cond_Deta_bded} make the boundaries $ \pm\infty $ inaccessible for the process $ X $, and it follows from \citet[p.  38]{mandl1968analytical} that the the transition probabilities $ p(t,x,\ud y) $ of $ X $ make $ T_t f(x) = \int_{\R} f(y) p(t,x,\ud y) $ a strongly continuous semigroup with the Feller property, namely $ T_t :C_b(\R) \to C_b(\R) $ for all $  t \geq 0 $. Let $ (\A_0, \Dom(\A_0)) $ be the closure of the infinitesimal generator of $ \{T_t: t\geq 0\} $. By \citet[Theorem 4.1, p 182]{ethier2009markov} $ X $ is generated by $ \A_0 $ in $ C_b(\R) $ and is the unique solution to the martingale problem for $ \A $. It follows that $ \A_0 = \A $ in $ \Dom(\A_0) =\{f \in \Dom(\A) : \A f \in C_b(\R) \} \subset \Dom(\A)$. Moreover, from standard semigroup theory, $ \deriv{}{t} T_t f = \A f $ for all $ f \in \Dom(\A_0) $, namely, $ p(t,x,\ud y) $ is the fundamental solution to the backwards equation $ \pderiv{v}{t} = \A v $. The result now follows from the uniqueness of fundamental solutions for parabolic differential equations \citep[see for example][]{friedman2013partial}
\end{proof}

We turn now to the construction of the diffusion $ X $ with speed and scale measures given by $ m $ and $ s $ in \eqref{Def_m_s}. The general procedure rests on the fact that the process $ s(X) $ is on natural scale and can be written as an appropriate re-scaling of a time-changed Brownian motion (see \citet{mandl1968analytical} for the general theory, or \citet{ramirez2011multi} for the case of piecewise constant coefficients). Then, we derive the stochastic differential equation that the process $ X $ should satisfy and verify that in fact, a strong solution exists. 

We first establish a useful lemma regarding the processes $ X $ and $ s(X) $.
\begin{lemma}\label{Lemma_ms}
	Let $ X $ be a diffusion process with scale and speed measures $ s $ and $ m $ defined through \eqref{Def_sHaHb}-\eqref{Def_mHaHb}, that are absolutely continuous with respect to Lebesgue measure, and have densities $ s' $ and $ m' $ respectively. 
	\begin{enumerate}
		\item Denote $ Y(t) = s(X(t)) $, $ t\geq 0$ where $ s $ is the scale function defined in \eqref{Def_scaleFun}. Then $ Y $  is a diffusion with scale function $ s_Y(x) = x $ and speed measure $ m_Y $ with density satisfying:
		\begin{equation}\label{Eq_relation_mmY}
		m'(x) = s'(x)  m_Y'(s(x)) \quad \text{a.e}.
		\end{equation}
		\item The quadratic variation of $ X $ is given by:
		\begin{equation}\label{Eq_QVX}
			\la X \ra_t = \int_0^t \frac{2}{m'(X(s)) s'(X(s))} \ud s
		\end{equation}
	\end{enumerate}
\end{lemma}

\begin{proof}
	That $ s_Y(x) = x $ follow because $ Y $ is on natural scale \citep[p. 302]{revuz_yor1999}.  Let $ a<b $ and recall definition \eqref{Def_mHaHb} of the speed measure. Denoting
	\begin{equation}
	g(a,b,x,y) := \frac{(x \wedge y-a)(b - x \vee y)}{b-a}, \quad x,y \in (a,b).
	\end{equation}
	we can write
	\begin{equation}\label{Eqn_proofLemma_sm1}
	\EXP_x \left(H^X_a \wedge H^X_b\right) = \int_a^b g(s(a),s(b),s(x), s(y)) m'(y) \ud y.
	\end{equation}
	Since $ x \to s(x) $ is an increasing function, $ \P_x(X(t) \in (a,b)) = \P_{s(x)}(Y(t) \in (s(a),s(b))) $ and the expected exit time in \eqref{Eqn_proofLemma_sm1} can also be written as
	\begin{align}
	\EXP_{s(x)} \left(H^Y_{s(a)} \wedge H^Y_{s(b)}\right) &= 
	\int_{s(a)}^{s(b)} g(s(a),s(b),s(x), y) m_Y'(y) \ud y\\
	&= \int_{a}^{b} g(s(a),s(b),s(x), s(z)) s'(z) m_Y'(s(z)) \ud z.
	\end{align}
	The uniqueness of the measure $ m $ \citep[Theorem 3.6, p. 304]{revuz_yor1999} implies \eqref{Eq_relation_mmY}.
	
	To prove the second assertion, let $ B $ be Brownian motion. It follows from \citep[Theorem 16.51]{breiman1992probability} that a version of $Y = s(X)$ can be written as a time change of $ B $  as follows: let $ \phi(t):= \int_0^t \frac{1}{2}m_Y'(B(r)) \ud r $ and $ T(t) = \phi^{-1}(t) $, then $ Y(t) = B(T(t)) $, $ t>0 $. In particular, $ \la Y \ra_t = T(t) $. The quadratic variation of $ X = s^{-1}(Y) $ is therefore
	 \begin{equation}\label{Eq_Lemma1_1}
	 \la X \ra_t = \int_0^t [(s^{-1})'(Y(r))]^2 \ud \la Y \ra_r = \int_0^t \frac{1}{[s'(X(r))]^2} \ud T(r).
	 \end{equation}
	 By \eqref{Eq_relation_mmY}, and performing a change of variables, we can also write $ T $ as
	 \begin{equation}\label{Eq_Lemma1_2}
	 T(t) =  \int_0^{T(t)} 2\frac{1}{m_Y'(B(r))} \ud \phi(r) = \int_0^t 2\frac{s'(X(r))}{m'(X(r))} \ud r.
	 \end{equation}
	 Combining equations \eqref{Eq_Lemma1_1} and \eqref{Eq_Lemma1_2} yields \eqref{Eq_QVX}.
	 
\end{proof}

 The following is an extension of results
of \citet{ouknine1991skew} for the case of piecewise constant coefficient
and a single interface, and is an equation of the general
type considered by \cite{leGall1984one} and, more recently,
\cite{bass2005one}.

\begin{theorem}\label{The_ConstructX}
	Under Assumptions \ref{assumptions}, the process $ X $ constructed in Theorem \ref{Theo_ExistenceX}  is the pathwise unique strong solution to 
	\begin{equation}\label{Eqn_SDE4X}
	 X(t) = \int_0^t \sqrt{\frac{D(X(s))}{\eta(X(s))}} \ud B(s) - \int_0^t \frac{D'(X(s))}{2 \eta(X(s))} \ud s 
	+ \sum_{j \in \Z} \frac{2 \lambda_j - 1}{2 \lambda_j} \Lp{X}(t,x_j)
	\end{equation}
	where $\Lp{X}(t,x)$ is right semimartingale local time of $X$ and the functions $ D $, $ \eta $ are taken to be left continuous at points in $ I $.
\end{theorem}

\begin{proof}
	 By Lemma \ref{Lemma_ms} the continuous martingale $Y(t) = s(X(t))$ has absolutely continuous quadratic variation $ \la Y \ra_t = \int_0^t Z^2(r) \ud r$ where $ Z(r) := \sqrt{\frac{2 s'(X(r))}{m'(X(r))}} $ . It follows from \citet[Theorem 3.4.2]{karatzas1991brownian} that there exists a filtered probability space with a Brownian motion $ B $, such that $ Y(t) = \int_0^t Z(r) \ud B(r) $. Since $ D $ is assumed to be of bounded variation, the function $ s^{-1}$ can be written as the difference of two convex functions and 
	\begin{equation}
	(s^{-1})''(\ud x) = - \frac{s''(s^{-1}(x))}{[s'(s^{-1}(x))]^3} \ud x + \sum_{j \in \Z} \left[\frac{1}{s'(x_j^+)} - \frac{1}{s'(x_j^-)}\right] \delta_{s^{-1}(x_j)}(\ud x).
	\end{equation}
	Applying the Ito-Tanaka and occupation times formulas \citep[Theorem 1.5 and Exercise 1.23, Chapter VI]{revuz_yor1999} on $ X(t) = s^{-1}(Y(t)) $ yields
	\begin{equation}\label{Eq_spp}
	\begin{split}
	X(t) = \int_0^t \sqrt{\frac{2}{s'_-(X(r)) m'_-(X(r))}} \ud B(r) &- \int_0^t \frac{s''(X(r))}{[s'(X(r))]^2 m'(X(r))} \ud r \\
	& + \sum_{j \in \Z} \left[1- \frac{s'(x_j^+)}{s'(x_j^-)}\right] \Lp{X}(t,x_j)
	\end{split}
	\end{equation}
	which coincides with \eqref{Eqn_SDE4X}.  The pathwise uniqueness of strong solutions follows from \cite{leGall1984one, bass2005one}, by noting that under the current assumptions, $ \sqrt{D/\eta} $ is a function of bounded variation, bounded away from zero, and the measure $ (s^{-1})''(\ud x) $ in \eqref{Eq_spp} is finite with $ \frac{2 \lambda_j - 1}{2 \lambda_j} < \frac{1}{2} $ for all $ j \in \Z $. 
\end{proof}

Having obtained the diffusion $X$ corresponding to the
conservation form of (\ref{EqFokkerPlanck}),  in the next
section we explore the role of continuity of flux in the structure of $X$ and its local time. 


\section{Various Notions of Local Time}\label{Sec_LocalTime}

Local time has a striking mathematical role in the 
development of the modern theory of stochastic processes,
from Brownian motion and diffusion, to continuous semimartingale
calculus, e.g., see \cite{revuz_yor1999, rogers2000diffusions, ito1974diffusion}.  In the course of this development two particular variations on the notion of local time have occured as follows:

\begin{definition}\label{Def_smlt}
  Let $X$ be a continuous semimartingale
with quadratic variation $\langle X\rangle_t$.  The {\it right, left,
symmetric semimartingale
local time (smlt)} of $X$ is a stochastic process, respectively denoted
 $L_+^X(t,x), L_-^X(t,x), L_*^X(t,x),$  $ x\in\R, $  $
t\geq 0,$ continuous in $t$ and determined
by either being right-continuous in $x$, left continuous in
$x$, or by averaging $L_*^X(t,x) = (L_+^X(t,x)+L_-^X(t,x))/2$, and
such that in any case
\begin{equation}
\int_0^t\varphi(X(s)) \ud \langle X\rangle_s = \int_\R\varphi(x)
L_{\pm, *}^X(t,x) \ud x,
\end{equation}
almost surely for any positve Borel measurable function $\varphi$.
\end{definition}

\begin{remark} The notation here is slightly different from that
of \cite{revuz_yor1999}. 
\end{remark}

Observe that by choosing $\varphi$ as an indicator function
of an interval $[x, x+\epsilon), \epsilon > 0,$ one has by
right-continuity, for example, that
\begin{equation}\label{Def_Lim_smlt}
L^X_+(t,x) = \lim_{\epsilon\downarrow 0}{1\over\epsilon}
\int_0^t \IND{[x,x+\eps)}(X(s))  \ud \langle X\rangle_s.
\end{equation}
Similarly,  $\Lm{X}$ can be obtained by using the indicator on the interval $ (x-\eps,x] $.

The next definition is that of diffusion local time (dlt) and requires Feller's notions of speed measure $m(\ud x)$ and scale function $s(x)$ (see \eqref{Def_sHaHb} and \eqref{Def_mHaHb}). It is customary to define dlt only for diffusions on natural scale; e.g., \cite{ito1974diffusion,rogers2000diffusions}. However, since any diffusion in natural scale is a time change of Brownian motion, it follows that its local times will therefore be themselves time changes of the local time of Brownian motion, and therefore always (spatially) continuous. On the other hand, for a general Feller diffusion $ X $ with scale function $ s $, the transformation $ Y = s(X) $ produces a diffusion on natural scale. This transformation renders local time continuity as a {\it generic} property that does not further inform more specific structure of the diffusion $X$. We thus extend the definition to Feller diffusions with any scale.

\begin{definition}\label{Def_dlt}   Let $Y$ be a diffusion 
with speed measure $m_Y(\ud y)$.  Then the {\it diffusion local time (dlt)} of $Y$, denoted $\tilde{L}^Y(y,t)$ is 
specified by 
\begin{equation}
\int_0^t\varphi(Y(r)) \ud r = \int_I\varphi(y)\tilde{L}^Y(y,t) \, m_Y(\ud y),
\end{equation}
almost surely for any positive Borel measurable function $\varphi$.
\end{definition}

By Lebesgue's differentiation theorem, it follows that
\begin{equation}\label{Def_Lim_dlt}
\tilde L_+^X (t,x) = \lim_{\epsilon\downarrow 0} \frac{1}{m[x,x+\eps)}
\int_0^t \IND{[x,x+\eps)}(X(s))  \ud s
\end{equation}
with the corresponding formula for $ \tilde L_-^X $, and $ \tilde L^X_* $ as the average.

For  the case of piecewise constant coefficients at a single interface, a particular {\it local time continuity} property at the interface was identified in \cite{appuhamillage2013skew}.  It was useful there to consider a modification of the more standard notions of semimartingale and diffusion local time to one referred to as {\it natural local time}. This was achieved there by exploiting explicit connections with skew Brownian motion indicated above. An extension to piecewise continuous coefficients and multiple interfaces is obtained in the present note.  

The following modification of the definition of local time will be seen
as useful in precisely calibrating jump size  relative to the interface parameters.  Just as in the case of semimartingale local time, one may consider right, left, and
symmetric versions. 

\begin{definition}\label{Def_nlt}  Let $X$ be a regular diffusion.   The 
{\it natural local time (nlt)} of $X$, right, left, and symmetric, respectively, denoted $\ell_{\pm,*}^X(t,x)$ is specified
by the occupation time formula
\begin{equation}
\int_0^t\varphi(X(s)) \ud s = \int_I\varphi(y)\ell_{\pm,*}^X(t,x) \ud x,
\end{equation}
for any positive Borel measureable functions $\varphi$.  The 
right and left versions are defined by the respective
right-continuous, left-continuous versions, while the symmetric
nlt is defined by the arithmetic average of these two.
\end{definition}

\begin{remark}
In its simplest terms, the modification to natural local time is made physically natural by examination of its {\it units}, namely $[{T\over L}]$, whereas
those of smlt are those of (spatial) length $[L]$, while dlt is dimensionless.  However, as previously noted, its essential feature resides in the implications of continuity properties relative to conservation laws. In particular, this puts a notion of stochastic local time on par with fundamentally important principles of
{\it concentration flux} and {\it conservation of mass} for pdes.
\end{remark}

The relationship between the three notions of local time above is summarized in the following proposition, the proof of which follows from a direct application of definitions \ref{Def_smlt}, \ref{Def_dlt} and \ref{Def_nlt}. 
\begin{proposition}
	Let $ X $ be a Feller diffusion with speed measure $ m $ and quadratic variation $ \la X \ra_t = \int_0^t q(X(s)) \ud s $, then 
	\begin{equation}\label{Eq_rel_lts}
	\ell^X(t,x) = {L^X(t,x)\over q(x)} = m'(x) \tilde{L}^X(t,x), \quad \text{ a.s.}
	\end{equation}
	with right and left versions obtained by considering the right and left continuous versions of $ q $ and $ m' $ respectively. 
\end{proposition}

The celebrated theorem of \citet{trotter1958property} on the (joint)
continuity of local time for Brownian motion is well-known for
the depth it provides to the analysis of Brownian motion paths.
This result is also naturally at the heart of the 
following general characterization of continuity of natural local time for regular diffusions.

\begin{theorem} \label{Theo_nlt_m_cont}
Let $X$ be a Feller diffusion on $\R$ with absolutely continuous
speed measure $m(\ud x) = m'(x) \ud x$ and scale function $s(x)$. Then the ratio $\ell^X(t,x)/m'(x)$ is continuous.
Moreover,
  the natural local time of $X$ is continuous at $x$ if and
only if $m^\prime$ is continuous at $x$.
\end{theorem}
\begin{proof}
	Let $ Y(t) = s(X(t)) = B(T(t)) $ as in the proof of lemma \ref{Lemma_ms}. Then $ L^X(t,x) = L^B(T(t),x) $, $ t \geq 0 $, $ x\in\R $. On the other hand, \eqref{Eq_rel_lts} together with \eqref{Eq_QVX} and \cite[Excercise VI.1.23]{revuz_yor1999}, give
	\begin{equation}
	\ell^X(t,x) = \frac{m'(x) s'(x)}{2}  L^X(t,x) = \frac{ m'(x) }{2}L^Y(t,s(x)) = 
	\frac{m'(x)}{2}  L^B(T(t),s(x))
	\end{equation}
	and the assertion follows from continuity of $ L^B $ and the scale function.
\end{proof}

It is noteworthy that, in general, the semimartingale
local time  of $X$ is {\it not} made continuous by division by $ m' $ as in Theorem \ref{Theo_nlt_m_cont}, while, as has been previously noted,  the diffusion local
time of the process transformed to natural scale is {\it always} continuous.  From the point of view of applications the theorem shows that natural local time
furnishes a microscopic {\it probe} to detect interfacial parameters $\eta$, when $\beta =1$, or $\beta$ when
$\eta = D$, respectively, in (\ref{EqFokkerPlanck}),
through location and size of its discontinuities. 

The desired Theorem \ref{Theo_MainResult} on the role of the continuity of flux
for the process $X$ defined by \eqref{Def_A_final} now
follows as a corollary.

\begin{proof}[Proof of Theorem \ref{Theo_MainResult} ]
	Let $ x_j \in I $. By Theorem \ref{Theo_nlt_m_cont} and the definitions of $ m' $,  $ \vp_j $ in \eqref{Def_phi_j}, \eqref{Def_m_s} give
	\begin{equation}
	\frac{\ell^X(t,x_j^+)}{\ell^X(t,x_j^-)} = \frac{m'(x_j^+)}{m'(x_j^-)} = \frac{\eta(x_j^+) \vp_{j-1}}{\eta(x_j^-)\vp_j} =  \frac{\eta(x_j^+) D(x_j^-) \lambda_j}{\eta(x_j^-)D(x_j^+) (1-\lambda_j)} 
	\end{equation}
\end{proof}

\section{Acknowledgments}
This research was partially supported by a grant DMS-1408947 from the National Science Foundation.

\bibliographystyle{abbrvnat}
\bibliography{ContLocalTime}

\end{document}